\documentclass{article}
\usepackage{authblk}
\usepackage{graphicx} 
\usepackage{babel}
\usepackage{amsmath}
\usepackage{amsfonts}
\usepackage{amsthm}
\usepackage{amssymb}
\usepackage{hyperref}
\usepackage{enumitem}

\newtheorem{thm}{Theorem}[section]
\newtheorem{pro}[thm]{Proposition}
\newtheorem{cor}[thm]{Corollary}
\newtheorem{obs}[thm]{Observation}
\newtheorem{lem}[thm]{Lemma}
\theoremstyle{definition}
\newtheorem{defn}[thm]{Definition}

\newcommand{\rad}{\textup{rad}}

\title{Idempotence in Carmichael and Giuga numbers and its application to the Agoh-Giuga conjecture}
\author{Miguel Ángel López\thanks{migulo23@ucm.es}}
\affil{Complutense University of Madrid}
\date{}

\begin{document}

\maketitle

\begin{abstract}
This article analyzes the idempotence of Giuga numbers and Carmi\-chael numbers. The first sections will lead us to define an arithmetic function through which we will find a new characterization of Giu\-ga numbers. The following sections will lead us to generalize Carmi\-chael num\-bers, finding numbers that are Fermat pseudoprimes in certain specific cases, and to formulate various statements and open questions regarding their proper\-ties. Finally, we will find a new alternative condition, based on idempo\-tence, that a number must satisfy to be a counterexample to the Agoh– \-Giuga conjecture and which, therefore, presumably, no number will ever satisfy. We will almost always work under the premise that $n$ is square-free, since both numbers always have that property, but in certain situa\-tions we will extend some of the results to the general case.
\end{abstract}

\section{Preliminary results}
The problem of the powers of elements of $\mathbb{Z}_n$ is one that has significant impli\-cations in the field of number theory. One of the earliest results, and still one of the most important today, is Fermat’s Little Theorem \cite{daepp2011fermat}, which states that in $\mathbb{Z}_p,$ where $p$ is a prime number, all powers of exponent $p-1$ are divisible by $p$ if $1$ is subtracted from them first. The result was refined by Euler in his famous theorem \cite{shiu2007euler} using the function $\varphi,$ one of the most celebrated arithmetic functions in number theory.

The study of these powers, however, still raises many questions. For example, the Agoh–Giuga conjecture, stated by Takashi Agoh as a property related to Bernoulli numbers \cite{agoh1995giuga}, but formulated by Giuseppe Giuga as a property of powers in $\mathbb{Z}_n$ \cite{giuga1950presumibile}, states that a number $p$ is prime if and only if $$1^{p-1}+2^{p-1}+\cdots + (p-1)^{p-1} \equiv -1 \pmod{p}, $$ and while the implication from left to right is trivial, precisely by using Fermat’s Little Theorem, the implication from right to left remains unproven to this day, since it has not been possible to verify that, for every composite number $n,$ the formula does not hold.

At the beginning of this article, we will study sums of powers in $\mathbb{Z}_n$ for the case where the exponent is precisely $\varphi(n)$, where $$\varphi(n) = \#\{m < n | \gcd(m,n) = 1\},$$ that is, the number of natural numbers less than $n$ that are coprime to $n$. In the second section, we will analyze the case where $n$ is square-free to arrive at a series of results that we will generalize, in the third section, to the general case. This result is closely related to idempotence in $\mathbb{Z}_n$ and will allow us, in the fourth section, to characterize Giuga numbers in a new way not previously defined, as well as to define a new arithmetic function motivated by the results obtained. Thanks to this function and the previous results, in the fifth and final section we will study the idempotents of Giuga and Carmichael numbers and characterize the Agoh-Giuga conjecture in an alternative way. To this end, in this first section, we will define certain concepts and preliminary results.

Henceforth, we will work in the ring $\mathbb{Z}_n$ of integers modulo $n$. This ring is a field if and only if $n$ is a prime number, and the following result, mentioned earlier, holds in it:

\begin{thm} (Fermat's Little Theorem)
Let $p$ be a prime number, $a\in\mathbb{Z}_p$, and $a\neq 0$; then $a^{p-1} \equiv 1 \pmod{p}.$
\end{thm}

Among other consequences, such as facilitating modular calculations on very large numbers, this theorem serves to analyze the structure of the subgroup generated by an element, as well as one of the most famous, though unsuccessful, attempts to characterize primality. Euler observed that the theorem could be generalized to $\mathbb{Z}_n,$ with $n$ composite, as follows:

\begin{thm}\label{euler} (Euler’s Theorem)
Let $n\in\mathbb{N}$, and let $a\in \mathbb{Z}_n$ such that $\gcd(a,n)=1$, then $a^{\varphi (n)} \equiv 1 \pmod{n}$, where $\varphi (n)$ is Euler’s $\varphi$-function.
\end{thm}

These theorems have immense implications, for example in the field of crypto\-graphy, as they form the basis of the RSA algorithm. However, Euler’s theorem can be and has been extended to all elements of $\mathbb{Z}_n$, not just those that are prime to $n$. To do this, we first need the concept of an idempotent element, which has been extensively used in algebra and number theory for the study of the ring $\mathbb{Z}_n$, since it holds, as we recall here and as can be read in \cite{sibley2012idempotents}, that the ring $\mathbb{Z}_n$ always has more idempotents than the trivial ones when $n$ is a composite number.

\begin{defn}
Let $n\in\mathbb{N}$; we say that $a$ is an idempotent element of $\mathbb{Z}_n$ if $a^2 \equiv a \pmod{n}$. Thus, $0$ and $1$ are always idempotent elements for every ring $\mathbb{Z}_n.$ These two idempotents are always present and are considered trivial.
\end{defn}

The first interesting consequence that follows is the following, which was proven and can be found in \cite{vass2015generalization}, although here we provide a quick proof of the result.

\begin{thm}
Let $n\in\mathbb{N}$ and $a\in \mathbb{Z}_n$, then $a^{\varphi (n)} \equiv b \pmod{n}$, where $b$ is an idempotent element of $\mathbb{Z}_n.$
\end{thm}

\begin{proof}
Let $n=p_1^{r_1}\ldots p_k^{r_k}$ be the factorial decomposition of $n.$ We will consider two cases, depending on whether each $p_i$ divides $a$ or not. If we first assume that $p_i\mid a,$ then, since $$r_i=1+(r_i-1)\ \leq 2^{r_i-1}\leq p_i^{r_i-1}\leq p_i^{r_i-1}(p_i-1)\leq\varphi (n),$$ we have that $a^{\varphi (n)} \equiv 0 \pmod{p_i^{r_i}}$. On the other hand, if $p_i\nmid a,$ then by Euler’s Theorem (\ref{euler}) and the fact that $\varphi \left(p_i^{r_i}\right)\mid\varphi (n),$ we obtain $a^{\varphi (n)} \equiv 1 \pmod{p_i^{r_i}}$. Therefore, in both cases $a^{\varphi (n)}(a^{\varphi (n)}-1) \equiv 0 \pmod{p_i^{r_i}},$ and thus $b$ is idempotent in $\mathbb{Z}_n$.
\end{proof}

\begin{thm}
Let $n \in \mathbb{N}$ and $n = p_1^{r_1} \ldots p_k^{r_k}$; then $\mathbb{Z}_n$ has exactly $2^k$ idempotent elements, including the trivial elements $0$ and $1$.
\end{thm}

\begin{proof}
The proof is straightforward and can be seen in detail in \cite{vass2016composite}; it uses the fact that every congruence equation of the form $a^2 \equiv a \pmod{p_i^{r_i}}$ has only two solutions, $0$ and $1$, therefore each idempotent of $\mathbb{Z}_n$ is uniquely associated, via the Chinese Remainder Theorem, with a sequence $(a_1,\ldots ,a_k)$ where $a_i\in \{0,1\}$. Thus, counting all possibilities, the total is $2^k.$
\end{proof}

\begin{cor}
The number of idempotents in $\mathbb{Z}_n$ is equal to $2^{\omega (n)},$ where $\omega(n)$ is the function that returns the number of prime divisors of $n$ (one if $n$ is a prime number). Therefore, the only idempotents of $\mathbb{Z}_p,$ where $p$ is prime, are $0$ and $1$.
\end{cor}

We will need some additional definitions and concepts, such as the orthogona\-lity of idempotent elements.

\begin{defn}
Let $n\in\mathbb{N}$, and let $a$ and $b$ be idempotents in $\mathbb{Z}_n$; then $a$ and $b$ are said to be orthogonal if $ab \equiv 0 \pmod n$.
\end{defn}

\begin{thm}\label{orta1-a}
Given an idempotent $a\in \mathbb{Z}_n$, we have that $1-a$ is also an idempotent and that $a$ and $1-a$ are orthogonal. In particular, all idempotents of $\mathbb{Z}_n$ are zero divisors in $\mathbb{Z}_n$ and therefore lack a multiplicative inverse in this same ring.
\end{thm}

\begin{proof}
Orthogonality is evident simply by doing the math, since $a(1-a) \equiv a^2 - a \equiv a - a \equiv 0 \pmod{n}$. It remains to be shown that $1-a$ is indeed an idempotent in $\mathbb{Z}_n$, but this is also straightforward since $(1-a)^2 \equiv 1+a^2-2a \equiv 1+a-2a \equiv 1-a \pmod{n}.$
\end{proof}

\section{Idempotence and the sum of powers, square-free case}

Now that these concepts have been defined, we can consider the first results. In this section, we will assume that $n$ is square-free and write its factorization as $n = p_1 \ldots p_k$. We will therefore have a total of $2^k$ idempotents in $\mathbb{Z}_n.$ It is important to note that, when we later apply results to Giuga and Carmichael numbers, the formulas in this section will suffice, since both types of numbers are square-free.

One of the fundamental results of this article, which will allow us to charac\-terize Giuga numbers in a new way, is the following.

\begin{thm}\label{ilc}
Let $n\in\mathbb{N}$ be a square-free number, the following holds:

\begin{enumerate}
\item The idempotent elements of $\mathbb{Z}_n$ are exactly those of the form $d^{\varphi (n)},$ where $d\mid n.$
\item The equation $$x^{\varphi (n)}\equiv \left(\frac{n}{d_0}\right)^{\varphi (n)}\pmod{n}$$ has exactly $\varphi (d_0)$ solutions in $\mathbb{Z}_n$.
\item $\sum_{a\in \mathbb{Z}_n}{a^{\varphi (n)}}\equiv\sum_{d\mid n}{\varphi (d)\cdot \left(\frac{n}{d}\right)^{\varphi (n)}}\pmod{n}.$
\end{enumerate}

\end{thm}

\begin{proof}

We know that if $n = p_1 \ldots p_k$, then $\mathbb{Z}_n$ has $2^k$ idempotents. We know that every element $a$ of $\mathbb{Z}_n$ that satisfies $a^{\varphi (n)}$ is idempotent; therefore, in particular, $\left(\frac{n}{d}\right)^{\varphi (n)}$ is idempotent, which is equivalent to saying that $d^{\varphi (n)}$ is idempotent. Furthermore, there are $2^k$ idempotents of this form, since they correspond to different ways of grouping divisors of $n$, which in total are $$\binom{n}{0}+\binom{n}{1}+\ldots \binom{n}{k}=2^k.$$ It suffices, therefore, to show that all these idempotents are distinct from one another.

Suppose that two of them are congruent, that is, $$d^{\varphi (n)}\equiv e^{\varphi (n)}\pmod{n},$$ with $d,e\mid n.$ Since $d \neq e,$ there exists, without loss of generality, at least one prime divisor $q$ of $d$ such that $q$ does not divide $e. $ Then $d\equiv 0 \pmod{q}$ and therefore $d^{\varphi (n)}\equiv 0 \pmod{q}.$ On the other hand, $\gcd(q,e)=1$ and by \ref{euler}, $e^{\varphi(q)} \equiv 1 \pmod{q}.$ Since $q \mid n$, we have that $\varphi(q) \mid \varphi(n)$ and $e^{\varphi(n)} \equiv 1 \pmod{q}. $ But we assumed that $d^{\varphi (n)}\equiv e^{\varphi (n)}\pmod{n},$ and therefore $d^{\varphi (n)}\equiv 1 \pmod{q},$ which is impossible. Therefore, $d^{\varphi (n)}\not\equiv e^{\varphi (n)}\pmod{n}.$

For the second part, we consider the equation $$x^{\varphi (n)}\equiv \left(\frac{n}{d_0}\right)^{\varphi (n)}\pmod{n}.$$ Suppose that $\frac{n}{d_0}=q_1\ldots q_{k_0}$ with $q_j\in\{p_1\ldots p_k\} .$ We have that $\left(\frac{n}{d}\right)^{\varphi (n)}$ is congruent to $0$ in $\mathbb{Z}_{p_i}$ if $p_i$ is equal to some $q_j$ and to $1$ otherwise, using the same reasoning as before. In the first case, the equation $$x^{\varphi (n)}\equiv \left(\frac{n}{d_0}\right)^{\varphi (n)}\pmod{p_i}$$ has a single solution; in the second case, we have that $$x^{\varphi (n)}\equiv x^{\varphi (p_i)\ \varphi\left(\frac{n}{p_i}\right)}\equiv 1\pmod{p_i}$$ using Fermat’s Little Theorem, and therefore, since $p_i$ is a prime number, there are exactly $p_i-1=\varphi (p_i)$ solutions. Applying this reasoning for each $p_i$, the equation has a number of solutions equal to $$\prod_{p_i\nmid\frac{n}{d_0}}\varphi (p_i)=\prod_{p_i\mid d_0}\varphi (p_i)=\varphi (d_0).$$

For the last part, in the previous section it was stated that every element $a$ of $\mathbb{Z}_n$ satisfies that $a^{\varphi (n)}$ is idempotent. Therefore, there exists a divisor $d$ of $n$ such that $$a^{\varphi (n)}\equiv\left(\frac{n}{d}\right)^{\varphi (n)}\pmod{n},$$ and there are $\varphi (d)$ of them in total. Since $a$ cannot generate two distinct idempotents, and $a^{\varphi (n)}\equiv 0 \pmod{n}$ if and only if $a\equiv 0 \pmod{n}$ because $n$ is square-free, it follows that $$\sum_{a\in \mathbb{Z}_n}{a^{\varphi (n)}}\equiv\sum_{d\mid n}{\varphi (d)\cdot \left(\frac{n}{d}\right)^{\varphi (n)}}\pmod{n}.$$
\end{proof}

\begin{cor}
If we define on $\mathbb{Z}_n$ the relation $a\sim b \iff a^{\varphi (n)}\equiv b^{\varphi (n)}\pmod{n},$ then this relation is trivially an equivalence relation and defines a quotient group, which we have shown to be precisely the group of idempotents of $\mathbb{Z}_n. $ Each class of this group can be represented as $\left[\left(\frac{n}{d}\right)^{\varphi (n)}\right],$ with $d\mid n.$ Furthermore, each class has exactly $\varphi (d)$ elements of $\mathbb{Z}_n.$
\end{cor}

We have performed this summation in a manner similar to how the expected value of a probability distribution is calculated, that is, by grouping the out\-comes of the variable, which are the idempotents, with their frequencies, which are the different values of $\varphi(d)$, where $d$ is a divisor of $n$. To do this, it has been fundamental to our reasoning that $n$ is square-free and, therefore, the only nilpotent in $\mathbb{Z}_n$ is zero. In the case of an arbitrary $n$, the existence of non-trivial nilpotents will be an additional obstacle we will have to deal with.

A consequence of the previous theorem is that it immediately shows that $d^{\varphi (n)}$ and $\left(\frac{n}{d}\right)^{\varphi (n)}$ are orthogonal. Given that $a$ and $1-a$ are also always orthogonal idempotents if $a$ is one, this leads us to wonder whether in this case we have uniqueness and both coincide. That is not true in both directions, as we will see later. That is, $a$ and $1-a$ are always orthogonal, but $a$ may have other orthogonal elements that do not take that specific form. However, the intuition that there exist idempotents that are both orthogonal and sum to 1 in $\mathbb{Z}_n$ leads us to a very simple but slightly more general result.

\begin{thm}\label{ab1}
Let $n=ab,$ with $\gcd(a,b)=1.$ Then $a^{\varphi (n)}+b^{\varphi (n)}\equiv 1 \pmod{n}. $ In particular, if $n$ is square-free, $d^{\varphi (n)}+\left(\frac{n}{d}\right)^{\varphi (n)}\equiv 1 \pmod{n}$ for every divisor $d$ of $n$.
\end{thm}

\begin{proof}
On the one hand, we have that $$a^{\varphi (n)}+b^ {\varphi (n)}\equiv b^{\varphi (n)} \equiv b^{\varphi (ab)} \equiv \left(b^{\varphi (a)}\right)^{\varphi (b)} \equiv 1 \pmod{a},$$ using \ref{euler} and the fact that $a,b$ are relatively prime. Similarly, $a^{\varphi (n)}+b^{\varphi (n)}\equiv 1 \pmod{b}.$ This system has a unique solution modulo $n$ by the Chinese Remainder Theorem, and therefore we must have $a^{\varphi (n)}+b^{\varphi (n)}\equiv 1 \pmod{n}.$
\end{proof}

The specific case of this last theorem is proven in \cite{isham2020arithmetic}, in Corollary 3.12. The theorem in this article is more general, as it is proven for any pair of values $a,b$ such that $n$ can be factored into a product of both and they have no common factors. This holds for orthogonal idempotents, but as we have seen, it is not a property they possess exclusively. The values of $a$ and $b$ are called unitary divisors, and are those divisors $d$ such that $\gcd \left( d,\frac{n}{d}\right)=1$. This property, in fact, can be generalized to products of more than two divisors that are pairwise coprime.

\begin{cor}
Using the same reasoning, it can be shown that if we take $n=a_1 \ldots a_k$ such that $\gcd(a_i,a_j)=1$ for all $1\leq i,j \leq k$, then $\sum_{i=1}^k{{a_i}^{\varphi(n)}}\equiv k-1 \pmod{n}.$
\end{cor}

\begin{cor}\label{semiprimo}
If $n$ is a semiprime, then $\sum_{i=1}^n{{i}^{\varphi(n)}} \equiv \varphi (n)-1 \pmod{n}.$
\end{cor}

\begin{proof}
If we consider $n=pq,$ in particular it is square-free, by Theorem \ref{ilc}, $$\sum_{i=1}^n{i^{\varphi(n)}}\equiv\varphi (1)n^{\varphi (n)}+\varphi (p)q^{\varphi (n)}+\varphi (q)p^{\varphi (n)}+\varphi (n)1^{\varphi (n)}\equiv$$ $$\equiv\varphi (p)q^{\varphi (n)}+\varphi (q)p^{\varphi (n)}+\varphi (n) \equiv \varphi (n)-q^{\varphi (n)}-p^{\varphi (n)} \pmod{n}. $$ By result \ref{ab1}, $p^{\varphi (n)}+q^{\varphi (n)} \equiv 1 \pmod{n}$ and therefore we have proven it.
\end{proof}

\begin{cor}\label{ilc2}
If $n=p_1\ldots p_k$ is square-free, then $$\sum_{i=1}^n{i^{\varphi (n)}}\equiv \sum_{p_i\mid n}{-\frac{n}{p_i}}\pmod{n}.$$
\end{cor}

\begin{proof}
We know from Theorem \ref{ilc} that $$\sum_{a\in \mathbb{Z}_n}{a^{\varphi (n)}}=\sum_{d\mid n}{\varphi (d)\cdot \left(\frac{n}{d}\right)^{\varphi (n)}}.$$ Let’s compute this sum modulo $p_i. $ We have that if $p_i\mid\frac{n}{d}$ then $\left(\frac{n}{d}\right)^{\varphi (n)} \equiv 0 \pmod{p_i},$ and if $p_i \nmid\frac{n}{d}$, then $\left(\frac{n}{d}\right)^{\varphi (n)} \equiv 1 \pmod{p_i}.$ Therefore

$$\sum_{d\mid n}{\varphi (d)\cdot \left(\frac{n}{d}\right)^{\varphi (n)}} \equiv \sum_{\substack{d\mid n \\ p_i\nmid\frac{n}{d}}}{\varphi (d)}\equiv$$

$$\equiv \sum_{\substack{d\mid n \\ p_i\mid d}}{\varphi (d)} \equiv \sum_{d\mid \frac{n}{p_i}}{\varphi (p_i)\varphi (d)} \equiv \sum_{d\mid \frac{n}{p_i}}{-\varphi (d)} \equiv -\frac{n}{p_i} \pmod{p_i}.$$

Therefore, by the uniqueness of the Chinese Remainder Theorem, we have already obtained the desired result.

\end{proof}

\begin{obs}\label{c-vs}It is important to note that Corollary \ref{ilc2} is a special case of a result by Carlitz and von Staudt, which can be found in \cite{carlitz1961staudt} and states that, if $k>1$ and $n$ is even, then

$$\sum_{i=1}^n{i^k}=\sum_{\substack{p\mid n \\ p-1\mid k}}{-\frac{n}{p}\pmod{n}}.$$

In our case, since $n$ is square-free and $k=\varphi(n),$ we have that $k$ is even since $\varphi(n)$ is always even for any value of $n$ greater than or equal to $3$ (although Corollary \ref{ilc2} also holds for $n=2$). On the other hand, it is immediately clear that $p-1\mid k$ for every divisor $p$ of $n$, and we therefore satisfy the conditions of the proposition.
\end{obs}

\begin{obs} This last result leads to one of the most widely used characte\-rizations of Giuga numbers, which is due to Giuga himself \cite{giuga1950presumibile}, and states that $n$ is a Giuga number if and only if 

$$\sum_{p\mid n}{\frac{1}{p}}-\prod_{p\mid n}{\frac{1}{p}}\in\mathbb{N}.$$

To arrive at this, we need only consider the equality 

$$\sum_{p_i\mid n}{-\frac{n}{p_i}}\equiv -1\pmod{n},$$

since $n$ is a Giuga number if $$\sum_{i=1}^n{i^{\varphi (n)}}\equiv -1\pmod{n},$$ and the calculations follow naturally. It may seem that we need to add the additional condition that $n$ be square-free, but, as we will see in the next section, the same result can be deduced for any value of $n$.
\end{obs}

\section{Idempotence and the sum of powers, gene\-ral case}

It is natural to wonder whether any of the previous results can be extended to the case where $n$ is not square-free. The answer is yes, but we must proceed with caution when proving them, since in the general case there are new elements that require attention: the so-called nilpotents of $\mathbb{Z}_n.$

\begin{defn}
Given $n\in\mathbb{N},$ we say that $a$ is a nilpotent of $\mathbb{Z}_n$ if there exists a positive $k$ such that $a^k\equiv 0 \pmod{n}.$
\end{defn}

Idempotents and nilpotents were both introduced by Benjamin Peirce in 1870 in \cite{peirce1881linear} and are a key tool for the study of powers in $\mathbb{Z}_n$, as well as the subject of doctoral theses even today, such as \cite{sanchez2023elementos}. One of the most classic results on nilpotents that we will use involves the so-called radical of a natural number.

\begin{defn}
Given $n\in\mathbb{N},$ $n=p_1^{r_1}\ldots p_k^{r_k},$ the radical of $n$ is defined as $m=p_1\ldots p_k,$ that is, the greatest square-free divisor of $n,$ and is denoted by $m=\rad (n).$
\end{defn}

\begin{thm}
The nilpotents of $\mathbb{Z}_n$ are the multiples of the radical of $n$, and in fact the radical of $n$ is the smallest non-trivial nilpotent of $\mathbb{Z}_n.$
\end{thm}

A first conclusion that follows directly from the previous section is that the idempotents of $\mathbb{Z}_n$ are obtained in a manner analogous to those of its square-free version; that is, they are the elements of the form $d^{\varphi (n)},$ but this time with $d\mid \rad (n). $ Another immediate conclusion from the previous results is that if $n$ is not square-free, the idempotents of $\mathbb{Z}_n$ are exactly the elements of the form $d^{\varphi (n)}$ with $d\mid n,$ $\gcd(d,\frac{n}{d})=1,$ that is, $d$ is a unit divisor. It is straightforward to prove this by simply repeating the steps of the proof for the square-free case in an almost identical manner.

In fact, it can be verified that, given any nonzero idempotent $a$ of the form $a\equiv d^{\varphi (n)}\pmod{n}$ with $d\mid\rad (n),$ then $a$ has exactly $2^{\omega (d)}$ idempotents orthogonal to it, since the smallest non-trivial orthogonal element of $d^{\varphi(n)}$ is equal to $\left(\frac{n}{d}\right)^{\varphi(n)},$ but any other of the form $\left( f\cdot\frac{n}{d}\right)^{\varphi(n)},$ with $f\mid d,$ will also be valid, since their product will still be a multiple of $\rad (n)$ and will be an idempotent distinct from the previous one, and in total there are as many of these as there are ways to combine the prime divisors of $d.$

For example, in $\mathbb{Z}_{30}$ the idempotents are $Id_{30}=\{0,1,6,10,15,16,21,25\}$ and all idempotents orthogonal to $15$ are $\{0,6,10,16\}, $ since $15\equiv\left(3\cdot 5\right)^8\pmod{30}, $ therefore $16\equiv\left(\frac{30}{15}\right) ^8\pmod{30}$ will be an idempotent orthogonal to it, but so will $10\equiv\left(2\cdot5\right) ^8\pmod{30},$ $6\equiv\left(2\cdot3\right) ^8\pmod{30}$, and
$0\equiv\left(2\cdot3\cdot5\right) ^8\pmod{30}$, which are $2^{\omega(15)}=4$ in total. However, $16$ is orthogonal only to $15$ and $0$ in $\mathbb{Z}_{30}$, since $16\equiv 2^8\pmod{30}$ and, therefore, since it is a prime number, $\omega(2)=1$. All the arguments are identical even if $n$ is not square-free.

For the following result, we will need two preliminary lemmas. It is a result that shows the periodicity of powers of the form $a^{\varphi(n),}$ where $\rad(n)$ is the period in the ring $\mathbb{Z}_n.$

\begin{lem}\label{lem1}
Let $n\in\mathbb{N}$ and $m=\rad(n)$. Then $m^s\varphi(n) = n\varphi(m^s)$.
\end{lem}

\begin{proof}
Let $n=p_1^{r_1}\ldots p_k^{r_k},$ then $$m^s\varphi (n) = \left(\prod_{i=1}^k{p_i^s}\right)\left(\prod_{i=1}^k{p_i^{r_i}-p_i^ {r_i-1}}\right) =$$ $$=\left(\prod_{i=1}^k{p_i^{r_i-1}\cdot p_i}\right)\left(\prod_{i=1}^k{p_i^{s-1}(p_i-1)}\right) = n\varphi (m^s).$$
\end{proof}

\begin{lem}\label{lem2}
Let $n\in\mathbb{N}$, $m=\rad (n). $ Suppose that $n=p_1^{r_1}\ldots p_k^{r_k},$ and let $s\in\mathbb{N}$ such that $n\nmid m^s.$ Then, if $r_i>s$, $p_i^{r_i-s}\mid \binom{\varphi (n)}{s}.$
\end{lem}

\begin{proof}
Suppose that $s=1,$ then $$\binom{\varphi (n)}{1}=\varphi (n)=\varphi \left (p_i^{r_i}\right)\cdot\varphi\left(\frac{n}{p_i^{r_i}}\right),$$ and $\varphi\left( p_i^{r_i}\right)=p_i^ {r_i-1}(p_i-1),$ so $p_i^{r_i-1}\mid \binom{\varphi (n)}{1}.$

In the general case, we have that

$$\binom{\varphi (n)}{s}=\frac{\varphi (n)-s+1}{s} \binom{\varphi (n)}{s-1}=\ldots =$$

$$= \frac{\left( \varphi (n)-s+1\right)\left( \varphi (n)-s+2\right)\ldots\left( \varphi (n)-1\right)}{s(s-1)\ldots\cdot 2}\cdot\binom{\varphi (n)}{1}.$$

By the base case, we know that $p_i^{r_i-1}\mid \binom{\varphi (n)}{1}$, so $p_i^{r_i-s}\mid p_i^{r_i-1}\mid \binom{\varphi (n)}{1}$. We have $p_i^{s-1}$ left over, but $p_i$ will divide $s(s-1)\ldots\cdot 2=s!$ several times. If we denote $x_s(p_i)$ as the multiplicity of $p_i$ in $s!$, a classic result of Legendre, which can be read in detail in \cite{aigner1999proofs}, tells us that

$$x_s(p_i)=\sum_{j\geq 1}{\left\lfloor\frac{n}{p_i^j}\right\rfloor},$$

but since $\lfloor x \rfloor < x,$ we have that $$\frac{x_s(p_i)}{s} < \frac{1}{p_i} + \frac{1}{p_i^2} + \ldots = \frac{1}{p_i-1}. $$ The worst-case scenario is $p_i=2,$ which leads to $x_s(2)<s,$ but since we have $s-1$ times the factor $p_i$ left over (twice in this case), we can cancel out all occurrences of $p_i$ in $s!$ with the ones left over from $p_i^{r_i-1}$ and obtain the desired result.
\end{proof}

Using these lemmas, we can now prove the following

\begin{thm}
Let $n\in\mathbb{N}$ and $m=\rad(n)$. Let $a\in\mathbb{Z}_n$ such that $a=dm+b$, where $d$ and $b$ are the unique quotient and remainder, respectively, of the Euclidean division. Then $a^{\varphi (n)}\equiv b^{\varphi (n)}\pmod{n}. $ In other words, powers of the form $u^{\varphi (n)}$ in $\mathbb{Z}_n$ are periodic with period $m,$ where furthermore $m^{\varphi (n)}\equiv 0 \pmod{n}$ and we have that $m$ is the smallest non-trivial nilpotent of $n.$
\end{thm}

\begin {proof}
We have that $$a^{\varphi (n)}=(dm+b)^{\varphi (n)}=\sum_{k=0}^{\varphi (n)}{\binom{\varphi (n)}{k}(dm)^k b^{\varphi (n)-k}}. $$ We will show that $$\binom{\varphi (n)}{k}m^k\equiv 0 \pmod{n}$$ for all $k\geq 1.$ Then all terms of the binomial will vanish except for $k=0,$ which is precisely $b^{\varphi (n)}.$

By Lemma \ref{lem1} we know that $m^k\varphi (n)=n\varphi (m^k)$ and therefore $$\binom{\varphi (n)}{k}m^k=\binom{\varphi (n)}{k}n\cdot\frac{\varphi (m^k)}{\varphi (n)}.$$ If $n\mid m^k$ then $\varphi(n)\mid \varphi(m^k)$ and we have it. Otherwise, there exists at least one $p_i$ with $r_i>k.$ We know that $p_i^{r_i}=p_i^{k}\cdot p_i^{r_i-k}$ and $p_i^k\mid m^k. $ By Lemma \ref{lem2}, $p_i^{r_i-k}\mid \binom{\varphi (n)}{k}.$ Therefore, $\varphi (n)\mid \binom{\varphi (n)}{k}\cdot\varphi (m^k)$, and we have proven the claim.
\end{proof}

\begin{cor}
Let $S_m=\sum_{i=1}^m{i^{\varphi (m)}}\pmod{m}$, where $S_m<m$. Then $$\sum_{i=1}^n{i^{\varphi (n)}}\equiv\frac{n}{m}\cdot S_m\pmod{n}.$$
\end{cor}

\begin{proof}
We know that $\mathbb{Z}_n$ has $\frac{n}{m}$ nilpotents, and each defines a sum-of-powers cycle of the form $i^{\varphi (n)}$, so it suffices to sum the first cycle and then multiply it by $\frac{n}{m}.$

We have that

$$S_m=\sum_{i=1}^m{i^{\varphi (m)}}\pmod{m}\equiv\sum_{i=1}^m{i^{\varphi (n)}}\pmod{m},$$

since $m\mid n$ and $i^{\varphi (m)}$ is an idempotent in $\mathbb{Z}_m.$ On the other hand, $S_m\equiv S_m+am$ in $\mathbb{Z}_n,$ with $a\geq 0$, and therefore $$\frac{n}{m}\cdot (S_m+am)\equiv\frac{n}{m}S_m+n\equiv\frac{n}{m}S_m\pmod{n},$$ and we can consider without loss of generality $a=0$ and $S_m<m.$

Then $$\frac{n} {m}S_m\equiv\frac{n}{m}\sum_{i=1}^m{i^{\varphi (m)}}\equiv\frac{n}{m}\sum_{i=1}^m{i^{\varphi (n)}}\equiv\sum_{i=1}^n{i^{\varphi (n)}}\pmod{n}.$$
\end{proof}

We will now verify that Theorem \ref{ilc} holds not only when $n$ is square-free, but also in the general case.

\begin{thm}\label{inlc}
For all $n\in\mathbb{N},$ $$\sum_{i=1}^n{i^{\varphi (n)}}\equiv\sum_{d\mid n}{\varphi (d)\left(\frac{n}{d}\right)^{\varphi (n)}}\pmod{n}.$$
\end{thm}

\begin{proof}
Let $p_j$ be a prime divisor of $n$. If $n=p_j^{r_j}u,$ with $\gcd(p_j,u)=1,$ then $i^{\varphi (n)}\equiv 0\pmod{p_j}$ if $p_j\mid i$ and $1$ otherwise, since in that case $\gcd(p_j,i)=1$ and $$\varphi (n)=\varphi (u)\cdot p_j^{r_j-1}(p_j-1)=\varphi (u)\cdot p_j^{r_j-1}\varphi (p_j),$$ and we can then apply Euler’s theorem. Let us verify that, for every prime divisor of $n$, the left and right sides of our hypothetical congruence coincide and, therefore, by the uniqueness of the Chinese Remainder Theorem, the congruence must necessarily hold.

The left side in $\mathbb{Z}_{p_j}$ is equivalent to

$$\#\left\{ i\in\mathbb{Z}_n / \gcd(p_j,i)=1\right\},$$

which is the same as

$$n-\frac{n}{p_j}\equiv -\frac{n}{p_j}\pmod{p_j}.$$

As for the right-hand side, $\left(\frac{n}{d}\right)^{\varphi (n)}$ is congruent in $\mathbb{Z}_{p_j}$ to $0$ if $p_j\mid\frac{n}{d}$ and to $1$ otherwise, so

$$\sum_{d\mid n}{\varphi (d)\left(\frac{n}{d}\right)^{\varphi (n)}}\equiv\sum_{p_j\nmid\frac{n}{d}}{\varphi (d)}\pmod{p_j}.$$

But this is the same as

$$\sum_{\substack{p_j^{r_j}\mid d\\d\mid n}}{\varphi (d)}=\sum_{a\mid \frac{n}{p_j^{r_j}}}{\varphi \left( ap_j^{r_j}\right)} =\sum_{a\mid \frac{n}{p_j^{r_j}}}{\varphi (a)\varphi \left( p_j^{r_j}\right)}=$$

$$=\varphi \left( p_j^{r_j}\right)\sum_{a\mid \frac{n}{p_j^{r_j}}}{\varphi (a)}=\varphi \left( p_j^{r_j}\right)\cdot\frac{n}{p_j^{r_j}}=\left ( p_j^{r_j}-p_j^{r_j-1}\right)\cdot\frac{n}{p_j^{r_j}}=n-\frac{n}{p_j},$$

which is congruent to $-\frac{n}{p_j}$ in $\mathbb{Z}_{p_j}$, and we are done.
\end{proof}

\begin{obs}
In the proof of Theorem \ref{inlc}, we also verified along the way—again using the Chinese Remainder Theorem—that the congruence in Theorem \ref{ilc2},

$$\sum_{i=1}^n{i^{\varphi (n)}}\equiv \sum_{p_i\mid n}{-\frac{n}{p_i}}\pmod{n},$$

holds for all $n\in\mathbb{N}$, not just for $n$ free of squares. In fact, once again this statement is a special case of the Carlitz-von Staudt lemma mentioned in observation \ref{c-vs}. We remain within the conditions of the proposition since $\varphi(n)$ is even if $n\geq 3$ (although we have also verified this for $n=2$) and $p_i-1\mid\varphi(n)$ for every prime divisor $p_i\mid n$, and therefore it is not necessary to specify this in the summation index.
\end{obs}

\section{A New Arithmetic Function and Characteri\-zation of Giuga Numbers}

Giuseppe Giuga defined the numbers bearing his name as those that are compo\-site and for which each of their prime factors $p_i$ satisfies $p_i\mid\left(\frac{n}{p_i}-1\right)$ or, equivalently, $p_i^2\mid (n-p_i)$ \cite{borwein1996giuga}. From this definition, it is easy to prove certain elementary properties, such as that a Giuga number must always be square-free and cannot be a semiprime either. The first Giuga numbers are $30,$ $858,$ $1722,$ $66198,$ $2214408306,$ and $24423128562,$ corresponding to sequence A007850 in the OEIS.

Giuga himself proved \cite{giuga1950presumibile} that a number $n$ is a Giuga number if and only if it satisfies the equation

$$\sum_{i=1}^n{i^{\varphi(n)}}\equiv-1\pmod{n},$$

and this relates these numbers very closely to everything we have developed here. For example, Corollary \ref{semiprimo} is an immediate and alternative proof that no semiprime number can be a Giuga number. We have also found a new characterization of these numbers thanks to Theorem \ref{inlc}:

\begin{cor}
Let $n\in\mathbb{N}$ be a composite number; then $n$ is a Giuga number if and only if

$$\sum_{d\mid n}{\varphi (d)\cdot \left(\frac{n}{d}\right)^{\varphi (n)}}\equiv-1\pmod{n}.$$

\end{cor}

All prime numbers trivially satisfy the congruence written above, and so Giuga numbers can in a sense be regarded as false candidates for a characteri\-zation of primality involving sums of powers. In fact, the Agoh–Giuga conjecture will be false if it is possible to find a Giuga number that is also a Carmichael number, and Carmichael numbers, which we will define later, are another well-known example of false candidates for prime numbers via a test, in this case Fermat’s Little Theorem \cite{riesel1994prime}. Both numbers also share several common proper\-ties, such as having at least three prime factors and being square-free, as well as a similar criterion for divisibility by factors.

The calculations we have performed motivate us to consider an arithmetic function that expresses these sums, which we will now define.

\begin{defn}
Given $n,k\in\mathbb{N}$, we define the arithmetic function $$G_k(n)=\sum_{d\mid n}{\varphi(n)\left(\frac{n}{d}\right)^k}.$$
\end{defn}

First, we have that $G_0(n)=\sum_{d\mid n}{\varphi(n)}=n$, and therefore $G_0\equiv Id$. On the other hand, using this new terminology, we have shown that

$$G_{\varphi(n)}(n)\equiv\sum_{i=1}^n{i^{\varphi(n)}}\pmod{n}.$$

Therefore, $n$ is a Giuga number if and only if $G_{\varphi(n)}(n)\equiv -1\pmod{n}.$

If we define the arithmetic function $Id_k(n)=n^k$, then $G_k=\varphi\ast Id_k$, where $\ast$ denotes the Dirichlet convolution. This implies that $G_k$ is multiplicative for $k\geq 0$, since the arithmetic functions $\varphi$ and $Id_k$ are multiplicative. Not only that, but also $G_k(1)=1$, which implies that $G_k$ has an inverse in the sense of this operation for $k\geq 0$. For more details on the Dirichlet convolution, see \cite{sandor2004handbook}.

Knowing also that $Id_k^{-1}=\mu\cdot Id_k$, we have that $\varphi=G_k\ast(\mu\cdot Id_k)$ for all $k\geq 0$. This result is also immediate if the Möbius inversion formula \cite{sandor2004handbook} is applied to the definition of $G_k$. For the special case $k=0$, we obtain

$$\varphi=Id\ast(\mu\cdot \textbf{1})=Id\ast \mu,$$

which is a classical result that can be found, for example, in \cite{sandor2004handbook}, and where we have used the notation that $\textbf{1}(n)=1$ for all $n\in\mathbb{N}$.

Since $G_k$ is a multiplicative function, we want to define it by applying it to elements of the form $p^r$, where $p$ is prime. If $r=1$, we have that

$$G_k(p)=\varphi(1)p^k+\varphi(p)1^k=p^k+p-1.$$

On the other hand, for $k,r>1$ we have that

$$G_k(p^r)=\sum_{d\mid p^r}{\varphi\left(\frac{p^r}{d}\right)d^k}=\sum_{j=0}^r{\varphi\left(p^{r-j}\right)\cdot p^{j^k}}=$$

$$=p^{rk}+\sum_{j=0}^{r-1}{\varphi\left(p^{r-j}\right)\cdot p^{jk}}=p^{rk}+(p-1)p^{r-1}\sum_{j=0}^{r-1}{\left(p^{k-1}\right)^j}=$$

$$=p^{rk}+(p-1)p^{r-1}\cdot\frac{p^{rk-r}-1}{p^{k-1}-1}.$$

For the case $k=1$, the penultimate step is equal to

$$p^r+(p-1)\sum_{j=0}^{r-1}{p^{r-1}}=p^r+r(p-1)p^{r-1}=p^r+r\varphi\left( p^r\right).$$

The case of greatest interest to us is that of $n=p_1\ldots p_s$ square-free, and for this case we have a very neat expression for the function given by

$$G_k(n)=\prod_{j=1}^s\left(p_j^k+p_j-1\right).$$

In particular, we will almost always focus on the case $k=\varphi(n)$, since thanks to the properties of idempotents we will obtain many identities.

This definition represents a new way of approaching square-free numbers and readily yields new results, but before that, we first formally define what a Carmichael number is. Essentially speaking, a Carmichael number is one of the relatively rare exceptions in which the converse of Fermat’s Little Theorem does not hold and, therefore, prevents the use of this theorem as an absolute primality test.

\begin{defn}
We say that a composite number $n\in\mathbb{N}$ is a Carmichael number if it satisfies the relation $b^n\equiv b\pmod{n}$ for all integers $b$.
\end{defn}

\begin{pro}
Let $n\in\mathbb{N}$ be a square-free number, and $p$ a prime divisor of $n$.
\begin{enumerate}
    \item $\left(G_{\varphi(n)}(p)\right)^2+G_{\varphi(n)}(p)\equiv p^2+p\pmod{n}.$
    \item $p\cdot G_{\varphi(n)}(p)\equiv p^2\pmod{n}.$
    \item $G_{\varphi(n)}(2n)\equiv 2\cdot G_{\varphi(n)}(n)\pmod{n},$ regardless of the parity of $n$.
    \item If $n$ is a Giuga number, $G_{\varphi(n)}\left(\frac{n}{p}\right)$ is a solution to the equation $px^2-x-1\equiv 0 \pmod{n}.$
    \item If $n$ is a Carmichael number, $G_{n-1}(n)\equiv G_{\varphi(n)}(n)\pmod{n}.$
\end{enumerate}
\end{pro}

\begin{proof} We solve each of the above statements one by one:

\begin{enumerate}
    \item $\left(G_{\varphi(n)}(p)\right)^2\equiv\left( p^{\varphi(n)}+p-1\right)^2\equiv \left(p^{\varphi(n)}\right)^2+p^2+1+2p\cdot p^{\varphi(n)}-2p^ {\varphi(n)}-2p\equiv p^{\varphi(n)}+p^2+1+2p-2p^{\varphi(n)}-2p\equiv -p^{\varphi(n)}+p^2+1\equiv p^2+p-\left( p^{\varphi(n)} +p-1\right)\equiv p^2+p-G_{\varphi(n)}(p)\pmod{n},$ where we have used the fact that $p^{\varphi(n)}$ is an idempotent in $\mathbb{Z}_n$ and acts as the identity element multiplied by $p$ in $\mathbb{Z}_n$.
    \item $p^{\varphi(n)}\cdot \left(p^{\varphi(n)}+p-1\right)\equiv \left(p^{\varphi(n)}\right)^2+p\cdot p^{\varphi(n)}-p^ {\varphi(n)}\equiv p\pmod{n}$, and therefore $p\cdot p^{\varphi(n)}\cdot G_{\varphi(n)}(p)\equiv p\cdot G_{\varphi(n)}(p)\equiv p^2\pmod{n}.$
    \item We note that $G_{\varphi(n)}(2)=2^{\varphi(n)}+1$ and that $G_{\varphi(n)}(4)=2^{2\varphi(n)}+2^{\varphi(n)}+2.$ If we assume $n=2\cdot p_2\ldots p_s$,

    $$G_{\varphi(n)}(2n)\equiv G_{\varphi(n)}(4)\cdot G_{\varphi(n)}(p_2\ldots p_s)\equiv$$ $$\equiv\left(2^{2\varphi(n)}+2^ {\varphi(n)}+2\right)\cdot G_{\varphi(n)}(p_2\ldots p_s)\equiv$$ $$\equiv 2\left(2^{\varphi(n)}+1\right)\cdot G_{\varphi(n)}(p_2\ldots p_s)\equiv 2\cdot G_{\varphi(n)}(n)\pmod{n}.$$

    For the odd case, it is trivial, since $n$ and $2$ have no common factors, that $G_{\varphi(n)}(2n)\equiv G_{\varphi(n)}(2)\cdot G_{\varphi(n)}(n)\equiv \left(2^{\varphi(n)}+1\right)\cdot G_{\varphi(n)}(n)\equiv 2 \cdot G_{\varphi(n)}(n)\pmod{n}$ by Euler’s Theorem.

    \item We know that $$1\equiv\left( G_{\varphi(n)}(n)\right)^2\equiv\left( G_{\varphi(n)}(p)\right)^2\left( G_{\varphi(n)}\left(\frac{n}{p}\right)\right)^2\equiv$$ $$\equiv\left( p+(p-1)\cdot G_{\varphi(n)}(p)\right)\cdot\left( G_{\varphi(n)}\left(\frac{n}{p}\right)\right)^2\pmod{n},$$ using the previous properties 1 and 2 and the fact that $n$ is a Giuga number. We therefore have that $$1\equiv p\cdot \left( G_{\varphi(n)}(\frac{n}{p})\right)^2-\left( G_{\varphi(n)}(\frac{n}{p})\right)\pmod{n}, $$ and thus $\left( G_{\varphi(n)}(\frac{n}{p})\right)$ is a solution to the equation $px^2-x-1\equiv 0\pmod{n}.$
    \item This is an immediate result if we use Korselt’s criterion \cite{wright2015variants}, which states that a number $n$ is a Carmichael number if and only if it is square-free and for every prime divisor of $n$, $(p-1)\mid (n-1).$ Knowing this,

$$G_{n-1}(n)\equiv \prod_{j=1}^s\left(p_j^{n-1}+p_j-1\right)\equiv -1\cdot \prod_{j=2}^s{\left(\left(p_j^{p_1-1}\right)^k+p_j-1\right)}\equiv$$ $$\equiv -1\cdot \prod_{j=2}^s{(1+p_j-1)}\equiv -\frac{n}{p_1}\pmod{p_1}.$$ Since we can do this for every $p_i\mid n,$ by Theorem \ref{ilc2} we have arrived at

$$G_{n-1}(n)\equiv -\sum_{p\mid n}{\frac{n}{p}}\equiv G_{\varphi(n)}(n)\pmod{n}.$$
\end{enumerate}
\end{proof}

\begin{cor}
The final result implies that if a number is both a Carmichael and a Giuga number, then it is a counterexample to the Agoh–Giuga Conjecture, a classic result in the study of this conjecture that has been obtained here via an alternative route. The converse of this result is not true even among square-free numbers. The first square-free number satisfying $G_{n-1}(n)\equiv G_{\varphi(n)}(n)\pmod{n}$ is 65, and it is not a Carmichael number. We will explore this issue further in the next section.
\end{cor}

\section{Idempotence in Giuga and Carmichael num\-bers. SC numbers and the Agoh–Giuga con\-jecture}

The way we have defined $G_{\varphi(n)}(n)$ shows that, for square-free values, it can be understood not only as an arithmetic function of the prime divisors of $n$, but also as a function of a special type of idempotents that we will now define. We will use these idempotents to relate them to Carmichael and Giuga numbers.

\begin{defn}
Let $n\in\mathbb{N}$. We say that an idempotent $e\in\mathbb{Z}_n$ is prime if it is of the form $e\equiv p^{\varphi(n)}\pmod{n}$, where $p$ is a prime number and a divisor of $n$.
\end{defn}

Within the subgroup of idempotents of $\mathbb{Z}_n$ under multiplication, the set $I_p(n)$ of prime idempotents of $\mathbb{Z}_n$ under multiplication generates all others except the identity element $1$, and does so in a minimal way; that is, if we remove an element from this set, it no longer generates them. On the other hand, although the definition can be given for any natural number, we are specifically interested in the square-free case.

We also have the following lemma, which allows us, given an idempotent of $\mathbb{Z}_n$, to recover the information regarding which divisor of $n$ generates it.

\begin{lem}\label{idp}
Let $n\in\mathbb{N}$, and let $e$ be an idempotent in $\mathbb{Z}_n$. Then $e\equiv d^{\varphi(n)}\pmod{n}$, where $d=\gcd(n,e)$. Furthermore, $d$ is unitary.
\end{lem}

\begin{proof}
We know that $e\equiv a^{\varphi(n)}\pmod{n}$ for a unitary divisor $a$ of $n$. Let us show that precisely $\gcd(n,e)=a$. First, $e-a^{\varphi(n)}=kn$, so $e=a^{\varphi(n)}+kn$ with $a\mid n$. Therefore, $a\mid e$ and $a\mid d=\gcd(n,e)$.

We can thus write $d=af,$ with $f\in\mathbb{N},$ $\gcd\left(f,\frac{n}{a}\right)=1.$ On the one hand,

$$d^{\varphi(n)}\equiv a^{\varphi(n)}f^{\varphi(n)}\equiv a^{\varphi(n)}\equiv e\pmod{\frac{n}{a}}$$

and furthermore $d^{\varphi(n)}\equiv 0\pmod{a}.$ On the other hand, $a^{\varphi(n)}\equiv e\pmod{\frac{n}{a}}$ and $a^{\varphi(n)}\equiv 0\pmod{a}. $ Therefore, by the uniqueness of the Chinese Remainder Theorem, since $\gcd\left(a,\frac{n}{a}\right)=1,$ we obtain

$$a^{\varphi(n)}\equiv d^{\varphi(n)}\pmod{n}.$$
            
\end{proof}

In the case of a square-free $n$, the prime idempotents are identified as those satisfying $\gcd(n,e)=p$, where $p$ is prime. By definition, the idempotents appearing in the explicit formula for $G_{\varphi(n)}(n)$ are all prime when $n$ is square-free, which leads us to the following alternative expression for such numbers:

$$G_{\varphi(n)}(n)=\prod_{e\in I_p(n)}{(e+\gcd(n,e)-1)}.$$

However, we can further specify the form of these idempotents in the case of Giuga numbers. 

\begin{thm}
Let $n\in\mathbb{N}$; $n$ is a Giuga number if and only if the prime idempotents of $\mathbb{Z}_n$ are the set 

$$I_p=\left\{1-\frac{n}{p}, p\text{ prime, }p\mid n\right\}$$
\end{thm}

\begin{proof}
We start with the implication from left to right. We know \cite{borwein1996giuga} that a number is a Giuga number if and only if $p\mid\left(\frac{n}{p}-1\right)$ for every prime divisor $p$ of $n$. Following a line of reasoning used in the recent article \cite{mondal2024idempotents}, where they classify all idempotents of $\mathbb{Z}_n$, since $\gcd\left(p,\frac{n}{p}\right)=1$ when $n$ is square-free (i.e., $p$ is a unit divisor), we have that

$$\frac{n}{p}\cdot p=n\mid \left[\left(\frac{n}{p}\right)^2-\left(\frac{n}{p}\right)\right]$$

and, therefore, $\frac{n}{p}$ is an idempotent in $\mathbb{Z}_n$. We know that $e=1-\frac{n}{p}$ is also an idempotent in $\mathbb{Z}_n$, and in fact is the prime idempotent associated with the divisor $p$. To see this, it suffices to check that $p\mid e$, which is true since, as $n$ is a Giuga number, $p\mid\left(\frac{n}{p}-1\right)$ and, furthermore, given that $q\not=p$ is a prime divisor of $n$, since $q\mid -\frac{n}{p}$, then $q\nmid e$ and, therefore, $gcd(n,e)=p$, which proves that $1-\frac{n}{p}$ is the prime idempotent associated with $p$ in $\mathbb{Z}_n$.

For the converse implication, we know that $1-\frac{n}{p}$ is a prime idempotent and, in particular, $\frac{n}{p}$ is also an idempotent in $\mathbb{Z}_n$. Then

$$\left(\frac{n}{p}\right)^2\equiv\left(\frac{n}{p}\right)\pmod{n}.$$ 

This implies that

$$\frac{n^2}{p^2}-\frac{n}{p}=kn$$

for some $k\in\mathbb{Z}$, and thus $n-p=kp^2$, which automatically implies that $p^2\mid (n-p)$, and this is equivalent to $p\mid\left(\frac{n}{p}-1\right).$ Therefore, $n$ is a Giuga number.

\end{proof}

\begin{cor}\label{giugacor}
$n\in\mathbb{N}$ is a Giuga number if and only if $\frac{n}{p}$ is an idempotent in $\mathbb{Z}_n$ for every prime divisor $p$ of $n$.
\end{cor}

\begin{cor}

If $n$ is a Giuga number, then $$G_{\varphi(n)}(n)\equiv\prod_{i=1}^s\left(\frac{p_i^2-n}{p_i}\right)\equiv -1\pmod{n},$$ but it is trivial to verify that the converse implication is not true.

\end{cor}

Corollary \ref{giugacor} allows us to study Giuga numbers from the perspective of the idempotents of $\mathbb{Z}_n$. In an article from the last century \cite{chernick1939fermat}, Jack Chernick proves that numbers of the form

$$n=(6k+1)(12k+1)(18k+1)$$

are Carmichael numbers if these three factors are all prime. Using idempo\-tents, it can be shown that none of these numbers (which are not yet known to be infinite, although it is believed that they are) is a Giuga number. 

To do this, we set $p=6k+1$, so that $$\frac{n}{p}=(12k+1)(18k+1)=216k^2+30k+1.$$

This value must be an idempotent of $\mathbb{Z}_n$, but if we check the calculations, we find that

$$\left(\frac{n}{p}\right)^2\equiv 2\left(\frac{n}{p}\right)\pmod{n}$$

regardless of the value of $k\in\mathbb{N},$ so these numbers, if they are Carmichael numbers, can never be Giuga numbers as well. 

These numbers can be extended to those of the form

$$Ch_m(k)=(6k+1)(12k+1)\prod_{i=1}^{m-2}{(9\cdot2^ik+1)},m\geq3,$$

with the condition that, for them to be Carmichael numbers, each factor must be prime and $k$ must be divisible by $2^{m-4}$ \cite{ribenboim2012new}. Following a similar line of reasoning, it can be shown that $$P(k)=\frac{Ch_m(k)}{6k+1}$$ should be an idempotent of $Ch_m(k)$, but this is not the case since $6k+1$ is not a factor of $P(k)$ or of $P(k)-1$, and therefore it cannot be an idempotent in $\mathbb{Z}_{6k+1}$ and, consequently, it cannot be one in $\mathbb{Z}_{Ch_m(k)}$ either.

Let’s try to understand what the prime idempotents of a Carmichael number are like. To do this, we have this simple proposition.

\begin{pro}\label{carm}
Let $n\in\mathbb{N}$ be a Carmichael number. The prime idempotents of $\mathbb{Z}_n$ are the numbers of the form $p^{n-1}$, where $p$ is a prime divisor of $n$.
\end{pro}

\begin{proof}
$\left(p^{n-1}\right)^2\equiv p^{2n-2}\equiv p^n\cdot p^{n-2}\equiv p\cdot p^{n-2}\equiv p^{n-1}\pmod{n},$ where we have used the fact that for a Carmichael number, we always have $b^n\equiv b\pmod{n}. $ On the other hand, it follows immediately from the definition and the fact that a Carmichael number is square-free that these idempotents are prime.
\end{proof}

\begin{cor}
The converse is not true, not even among square-free numbers. $65$ is the smallest square-free number satisfying $65=5\cdot 13$, and its prime idempotents are $5^{64}\equiv 40\pmod{65}$ and $13^{64}\equiv 26\pmod{65}.$
\end {cor}

In fact, for Proposition \ref{carm}, we do not need a number to be a Carmichael number, but rather to satisfy Fermat’s property at least for its divisors. This suggests the definition of numbers that satisfy a slightly more relaxed condition.

\begin{defn}
Let $n\in\mathbb{N}$ be a composite number; we say that $n$ is a Semi-Carmichael number (hereinafter referred to as SC number) if, for every divisor $d$ of $n$, we have $d^n\equiv d\pmod{n}.$
\end{defn}

Two trivial remarks are that every Carmichael number is also a SC number, and that for $n$ to be a SC number, it suffices that $d^n\equiv d\pmod{n}$ for the prime divisors of $n$. Other more interesting properties are:

\begin{pro}\label{semi}
Let $n$ be a SC number.
    \begin{enumerate}
        \item $n$ is square-free.
        \item Its prime idempotents are the numbers of the form $p^{n-1}$, where $p$ is a prime divisor of $n$, and the converse is also true if we require that $n$ be composite.
        \item $G_{n-1}(n)\equiv G_{\varphi (n)}(n)\pmod{n}$.
    \end{enumerate}
\end{pro}

\begin{proof}
To prove that $n$ is square-free, suppose that $n=p^rm$ with $r>1$ and $\gcd(m,p)=1$. By the SC number definition, $p^n \equiv p \pmod{n}$ and, since $r>1$, $p^n \equiv p \pmod{p^2}$. Since $n$ is composite and therefore greater than or equal to $2$, we must have $p \equiv 0 \pmod{p^2}$, which is impossible. Therefore, $r = 1$ for every prime divisor $p$ of $n$.

For the second part, given a SC number $n$, we have, as we already showed, that

$$\left( p^{n-1}\right)^2\equiv p^{2n-2}\equiv p^n\cdot p^{n-2}\equiv p^{n-1}\pmod{n}$$

and, since $n$ is square-free, $\gcd(n,p^{n-1})=p$ and therefore $p^{n-1}$ is a prime idempotent.

For the converse, let $n\in\mathbb{N}$ such that its prime idempotents are of the form $p^{n-1}$, where $p$ is a prime divisor of $n$. Then

$$\left(p^{n-1}\right)^2\equiv p^{2n-2}\equiv p^{n-1}\pmod{n},$$

and this implies that

$$p^n\cdot p^{n-1}\equiv p^n\pmod{n}.$$

If we consider a divisor $q$ of $n$ coprime with $p$, then $p^{n-1}$ has an inverse in $Z_q$ and we can state that

$$p^{n-1}\equiv 1\pmod{q}\implies p^n\equiv p\pmod{q},$$

and since it is trivial that $p^n\equiv p\pmod{p}$, using the Chinese Remainder Theorem we have that $p^n\equiv p\pmod{n}$ for every prime divisor $p$ of $n$ and, therefore, since $n$ is composite, it is also an SC number.

For the last statement, it suffices to use the previous section and Lemma \ref{idp} to show that

$$G_{n-1}(n)\equiv \prod_{e\in I_p(n)}{e+\gcd(e,n)-1}\equiv G_{\varphi(n)}(n)\pmod{n}.$$
\end{proof}

The converse of this last statement is not true; for example, $645 = 3 \cdot 5 \cdot 43$ satisfies $G_{n-1}(n) \equiv G_{\varphi(n)}(n) \pmod{n}$ and is not a SC number (although it does satisfy $43^{645} \equiv 43 \pmod{645}$). SC numbers do not satisfy even a partial version of Korselt’s criterion, for although many of them have prime divisors such that $(p-1) \mid (n-1)$, and in many cases all but one of their divisors satisfy this condition, the number $2501=41 \cdot 61$ is a SC number but $40$ and $60$ doesn't divide $2500$.

The first SC number is $65$ and the following ones are:

$$133,301,561,793,1105,1729,2041,2413,2465\ldots$$

The first SC numbers that have at least three prime factors and are not Carmichael numbers are:

$$7107,89281,91001,139897,176149,204001\ldots$$

\begin{cor}\label{corguigacarmichael}
Let $n\in\mathbb{N}$ be a composite number; $n$ is a Giuga and SC number if and only if for every prime divisor $p$ of $n$, 
    
$$p^{n-1}\equiv1-\frac{n}{p}\pmod{n}.$$
   
\end{cor}

\begin{proof}
The implication from left to right follows immediately from what we have seen so far, since for every prime divisor $p$ of $n$, its associated prime idempotent is, on the one hand, $1-\frac{n}{p}$ because it is a Giuga number divisor, and also $p^{n-1}$ because it is a SC number divisor; therefore, both values are equal modulo $n$.

Regarding the converse implication, given $p$ a prime divisor of $n$,
    
$$\left(p^{n-1}\right)\left(1-\frac{n}{p}\right)\equiv p^{n-1}-np^{n-2}\equiv p^{n-1}\pmod{n},$$

where $n-2\geq1$ since $n$ is composite, so $1-\frac{n}{p}$ is an idempotent for every prime divisor $p$ of $n$. By Corollary \ref{giugacor}, $n$ is a Giuga number and, in particular, is square-free. On the other hand,

$$p\cdot p^{n-1}\equiv p^n\equiv p-n\equiv p\pmod{n},$$

and since $n$ is composite and this holds for every prime divisor of $n$, then $n$ is a SC number, although this is also immediate since we have shown in \ref{semi} that every composite number whose prime idempotents are of the form $p^{n-1}$ is a SC number, and these idempotents are prime since $n$ is Giuga and consequently square-free, and $\gcd(n,p^{n-1})=p$.

\end{proof}

We have therefore defined an alternative property that a number $n$ must satisfy in order to be a counterexample to the Agoh–Giuga conjecture; thus, it would suffice to show that no number satisfies the congruence \ref{corguigacarmichael} for all its prime divisors in order to prove this conjecture.

\bibliographystyle{plain}
\bibliography{bibliography}

\end{document}